\documentclass[12pt]{article}
\usepackage{amsmath, amsthm}
\usepackage{enumerate}
\usepackage{amssymb}
\usepackage{geometry}
\newtheorem{thm}{Theorem}[section]
\newtheorem{exam}{Example}[section]
\newtheorem{cor}[thm]{Corollary}

\newtheorem{lem}[thm]{Lemma}
\newtheorem{prop}[thm]{Proposition}

\theoremstyle{definition}
\newtheorem{defn}{Definition}[section]
\theoremstyle{remark}
\newtheorem{rem}{Remark}[section]

\DeclareMathOperator{\ptc}{\xrightarrow[]{p_\tau}}
\DeclareMathOperator{\uptc}{\xrightarrow[]{up_\tau}}
\DeclareMathOperator{\tc}{\xrightarrow[]{\tau}}

\DeclareMathOperator{\unc}{\xrightarrow[]{un}}
\DeclareMathOperator{\uoc}{\xrightarrow[]{uo}}

\DeclareMathOperator{\upc}{\xrightarrow[]{up}}

\DeclareMathOperator{\oc}{\xrightarrow[]{o}}

\DeclareMathOperator{\pc}{\xrightarrow[]{p}}
\DeclareMathOperator{\nc}{\xrightarrow{\lVert \cdot \rVert}}

\newcommand{\eval}[2][\right]{\relax
  \ifx#1\right\relax \left.\fi#2#1\rvert}

\begin{document}

\title{\bf Unbounded $p_\tau$-Convergence in Vector Lattices Normed by Locally Solid Lattices} 
\maketitle

\author{\centering Abdullah AYDIN \\ \bigskip  \small  
	Department of Mathematics, Mu\c{s} Alparslan University, Mu\c{s}, Turkey. \\}

\bigskip

\abstract{Let $(x_\alpha)$ be a net in a vector lattice normed by locally solid lattice $(X,p,E_\tau)$. We say that $(x_\alpha)$ is unbounded $p_\tau$-convergent to $x\in X$ if $p(\lvert x_\alpha-x\rvert\wedge u)\xrightarrow{\tau} 0$ for every $u\in X_+$. This convergence has been studied recently for lattice-normed vector lattices as the $up$-convergence in \cite{AGG,AEEM,AEEM2}, the $uo$-convergence in \cite{GTX}, and, as the $un$-convergence in \cite{DOT,GX,GTX,KMT,Tr2}. In this paper, we study the general properties of the unbounded $p_\tau$-convergence.
}

\bigskip
\let\thefootnote\relax\footnotetext
{
	
\text{2010 AMS Mathematics Subject Classification:} Primary 46A40; Secondary 46E30.

e-mail: a.aydin@alparslan.edu.tr}

\section{Introduction}
Locally solid vector lattices and lattice-valued norms provide a natural and efficient tools in the functional analysis. We refer the reader for detail information; see for example \cite{AA,AB,ABPO,A,DEM,E,K,L}. In this paper, the aim is to illustrate the usefulness of lattice-valued norms for investigation of different types of {\em unbounded $p$-convergence} in a lattice-normed vector lattice; see \cite{AGG,AEEM,AEEM2} and different types of {\em unbounded convergence} in a vector lattice, which attracted the attention of several authors in series of recent papers; see for example \cite{A,DOT,GX,GTX,KMT,Tr2}. 

Recall that a net $(x_\alpha)_{\alpha\in A}$ in a vector lattice $X$ is {\em order convergent} to $x\in X$ if there exists another net $(y_\beta)_{\beta\in B}$ satisfying $y_\beta \downarrow 0$, and for any $\beta\in B$, there is $\alpha_\beta\in A$ such that $\lvert x_\alpha-x\rvert\leq y_\beta$ for all $\alpha\geq\alpha_\beta$. In this case, we write $x_\alpha\oc x$. In a vector lattice $X$, 
a net $(x_\alpha)$ is {\em unbounded order convergent} to $x\in X$ if $\lvert x_\alpha-x\rvert\wedge u\oc 0$ for every $u\in X_+$; see for example \cite{DOT,GTX,KMT}. In this case, we write $x_\alpha\uoc x$. The $uo$-convergence is an abstraction of the $a.e.$-convergence in $L_p$-spaces for $1\leq p<\infty$, \cite{GX,GTX}. In a normed lattice $(X,\lVert \cdot\rVert)$, a net $(x_\alpha)$ is {\em unbounded norm convergent} to $x\in X$, written as $x_\alpha\unc x$, if $\lVert \lvert x_\alpha-x\rvert\wedge u\rVert\to 0$ for every $u\in X_+$; see \cite{DOT}. 

Let $X$ be a vector space, $E$ be a vector lattice, and $p:X\to E_+$ be a vector norm (i.e. $p(x)=0\Leftrightarrow x=0$; $p(\lambda x)=|\lambda|p(x)$ for all $\lambda\in\mathbb{R}$, $x\in X$; and $p(x+y)\leq p(x)+p(y)$ for all $x,y\in X$) then the triple $(X,p,E)$ is called a {\em lattice-normed space}, abbreviated as $LNS$; see for example \cite{K}. If $X$ is a vector lattice and the vector norm $p$ is monotone (i.e. $\lvert x\rvert\leq \lvert y\rvert\Rightarrow p(x)\leq p(y)$) then the triple $(X,p,E)$ is called a {\em lattice-normed vector lattice}, abbreviated as $LNVL$; see \cite{AEEM}. We abbreviate the convergence $p(x_{\alpha}-x)\oc 0$ as $x_\alpha\pc x$, and say in this case that $(x_\alpha)$ {\em $p$-converges} to $x$. A net $(x_\alpha)_{\alpha \in A}$ in an $LNS$ $(X,p,E)$ is said to be {\em $p$-Cauchy} if the net $(x_\alpha-x_{\alpha'})_{(\alpha,\alpha') \in A\times A}$ $p$-converges to $0$. An $LNS$ $(X,p,E)$ is called (\textit{sequentially}) {\em $p$-complete} if every $p$-Cauchy (sequence) net in $X$ is $p$-convergent. In an $LNS$ $(X,p,E)$, a subset $A$ of $X$ is called {\em $p$-bounded} if there exists $e\in E$ such that $p(a)\leq e$ for all $a\in A$. An $LNVL$ $(X,p,E)$ is called {\em $op$-continuous} if $x_\alpha\oc 0$ implies $p(x_\alpha)\oc 0$. A net $(x_\alpha)$ in an $LNVL$ $(X,p,E)$ is said to be {\em unbounded $p$-convergent} to $x\in X$ (shortly, $(x_\alpha)$ $up$-converges to $x$ or  
$x_\alpha\upc x$), if $p(|x_\alpha-x|\wedge u)\oc 0$ for all $u\in X_+;$ see \cite[Def.6]{AEEM}.   

Let $E$ be a vector lattice and $\tau$ be a linear topology on $E$ that has a base at zero consisting of solid sets. Then the pair $(E,\tau)$ is called a {\em locally solid vector lattice}. It follows from \cite[Thm.2.28]{AB} that a linear topology $\tau$ on a vector lattice $E$ is locally solid vector lattice iff it is generated by a family of Riesz pseudonorms $\{\rho_j\}_{j\in J}$. Moreover, if a family of Riesz pseudonorms generates a locally solid topology $\tau$ on a vector lattice $E$, then $x_\alpha \tc x$ iff $\rho_j(x_\alpha-x)\to 0$ in $\mathbb{R}$ for each $j\in J$.  It should be noted that all topologies considered throughout this article are assumed to be Hausdorff, and also, unless otherwise, the pair $(E,\tau)$ refers to a locally solid vector lattice with a family of Riesz pseudonorms $\{\rho_j\}_{j\in J}$ that generates the topology $\tau$. Moreover, all vector lattices are assumed to be real and Archimedean. For any $u\in E$, one can observe that $\frac{1}{n}u\tc 0$ in a $(E,\tau)$ because of $\rho_j(\frac{1}{n}u)\leq\frac{1}{n}\rho_j(u)\to 0$ in $\mathbb{R}$ for all $j\in J$; see \cite[p.2]{DEM}. So, we shall keep in mind the following lemma, obtained from \cite[Thm.2.28]{AB}. 

\begin{lem}\label{monoton is convergent}
Let $(x_\alpha)_{\alpha\in A}$ and $(y_\alpha)_{\alpha\in A}$ be two nets in a locally solid vector lattice $(E,\tau)$. If $\lvert y_\alpha\rvert\leq \lvert x_\alpha\rvert$ for all $\alpha \in A$ and $x_\alpha\tc 0$ in $E$ then $y_\alpha\tc 0$ in $E$.  
\end{lem}

Let $(X,p,E)$ be an $LNVL$ whit $(E,\tau)$ being a locally solid vector lattice, then $(X,p,E_\tau)$ is called a {\em vector lattice normed by locally solid lattice}, abbreviated as $LSNVL$. Dealing with $LSNVL$s, we shall keep in mind also the following examples.

\begin{exam}
Let $(E,\tau)$ be a locally solid vector lattice. Then $(E,\lvert \cdot \rvert, E_\tau)$ is an $LSNVL$.
\end{exam}

\begin{exam}
Let $(X,\lVert \cdot \rVert)$ be a normed lattice. Then $(X,\lvert\cdot\rvert,X_{\lVert\cdot\rVert})$ is an $LSNVL$.
\end{exam}
By considering Example 2.4 of \cite{L}, we give the following example.
\begin{exam}
Let $E$ be the space of all Lebesgue measurable functions on $\mathbb{R}$ with the usual pointwise ordering, i.e., for $f,g\in E$, we define $f\leq g$ iff $f(t)\leq g(t)$ for every $t\in \mathbb{R}$. Consider the map $\lVert \cdot \rVert:E\to \mathbb{R}$ defined by $\lVert f \rVert =\big(\int (f(t))^2dt\big)^{\frac{1}{2}}$, where $f\in E$. Then the norm $\lVert \cdot \rVert$ is a seminorm on $E$. It is easy to see that it is also a Riesz seminorm. Thus, the topology $\tau$ generated by $\lVert \cdot \rVert$ is locally convex-solid; see \cite[Thm.2.25]{AB}. Now, consider a vector lattice $X$ and a map $p:X\to E$ defined by $p(x)=p(x)[f]=\lvert f\rvert(\lvert x\rvert)$. Then $(X,p,E_\tau)$ is an $LSNVL$.
\end{exam}


\section{The $p_\tau$-Convergence}
Most of the notions and results in this section are direct analogs of well-known facts of the theory of normed lattices and locally solid vector lattices. 
\begin{defn}
Let $(X,p,E_\tau)$ be an $LSNVL$. A net $(x_\alpha)$ in $X$ is called {\em $p_\tau$-convergent} to $x\in X$ if $p(x_\alpha-x)\tc 0$ in $E$. We write $x_\alpha\xrightarrow{p_\tau}x$, and say that $(x_\alpha)$ $p_\tau$-converges to $x$.
\end{defn}

In the following lemma, we give some basic properties of the $p_\tau$-convergence.
\begin{lem}\label{basic properties}
Let $(x_\alpha)$ and $(y_\beta)$ be two nets in an $LSNVL$ $(X,p,E_\tau)$, then;
\begin{enumerate}
\item[(i)] $x_\alpha\ptc x$ iff $(x_\alpha- x)\ptc 0$;

\item[(ii)] if $x_\alpha\ptc x$ then $y_\beta\ptc x$ for each subnet $(y_\beta)$ of $(x_\alpha)$;

\item[(iii)] suppose $x_\alpha\ptc x$ and $y_\beta\ptc y$, then $ax_\alpha+by_\beta\ptc ax+by$ for any $a,b\in \mathbb{R}$;

\item[(iv)] if $x_\alpha \ptc x$ and $x_\alpha \ptc y$ then $x=y$; 

\item[(v)] if $x_\alpha \ptc x$ then $\lvert x_\alpha\rvert \ptc \lvert x \rvert$.
\end{enumerate}
\end{lem}

\begin{proof}
$(i)$, $(ii)$ and $(iii)$ are obvious. Observe the inequality $p(x-y) \leq p(x-x_\alpha)+p(x_\alpha-y)$, and so $(iv)$ holds. For $(v)$, by using the inequality $\big\lvert \rvert x_\alpha\rvert -\lvert x\rvert\big\rvert\leq \lvert x_\alpha -x\rvert$, we have $p(\rvert x_\alpha\rvert -\lvert x\rvert)\leq p(x_\alpha -x)$. Therefore, we get the result.
\end{proof}

The $p_\tau$-convergence coincides with some kinds of convergence, and we show them in the following remarks. As it was observed in \cite{DOT}, the $un$-convergence is topological. This topology is called a {\em $un$-topology}, see \cite[p.3]{KMT}, and also, it is a locally solid topology. 
\begin{rem}\label{un convergence}
Consider an $LSNVL$ $(X,p,E_\tau)$ with $E$ being a Banach lattice and the topology $\tau$ being $un$-topology on $E$. Then, for any net $(x_\alpha)$ in $X$,
\begin{enumerate}
\item[(1)] $x_\alpha\ptc x$ in $X$ iff $p(x_\alpha-x)\unc 0$ in $E$,
		
\item[(2)] if $E$ has a strong unit, then $x_\alpha\ptc x$ in $X$ iff $p(x_\alpha-x) \nc 0 $ in $E$; see \cite[Prop.2.3]{KMT},
		
\item[(3)] if $E$ has a quasi-interior point, then $x_\alpha\ptc x$ in $X$ iff  $p(x_\alpha-x)\xrightarrow{d} 0$ in $E$, where $d$ is a metric on $E$ which gives the $un$-topology; see \cite[Prop.3.2]{KMT}.
\end{enumerate}
\end{rem}

For an $LSNVL$ $(X,p,E_\tau)$ with $(E,\lVert\cdot\rVert)$ being a normed vector lattice and $\tau$ being the topology generated by a family of Riesz pseudonorms $p_u:E\to \mathbb{R}_+$, for each $u\in E_+$, defined by $p_u(x)=\big\lVert \lvert x\rvert\wedge u\big\rVert$. Then, for a net $(x_\alpha)$ and $x\in X$, we have $x_\alpha \ptc x$ iff $p(x_\alpha-x) \xrightarrow{un} 0$; see \cite[Thm.2.1]{EV}.

Let $M=\{m_\lambda\}_{\lambda \in \Lambda}$ be a separating family of lattice semi-norms on a vector lattice $E$. Then a net $(x_\alpha)$ in $E$ is $um$-convergent to $x$ if $m_\lambda(x_\alpha-x)\wedge u\to 0$ for all $\lambda \in \Lambda$ and $u\in E_+$. This convergence is topological, and that is called an $um$-topology; see \cite[p.4]{DEM2}.
\begin{rem}\label{um convergence}
Let $(X,p,E_\tau)$ be an $LSNVL$ with $(E,\tau)$ is $um$-topology. Then, for a net $(x_\alpha)$ and $x\in X$, we have $x_\alpha \ptc x$ iff $p(x_\alpha-x) \xrightarrow{um} 0$ in $E$.
\end{rem}

Now, we continue with some basic results about the $p_\tau$-convergence.
\begin{prop}\label{triangle convergence}
Let $(X,p,E_\tau)$ be an $LSNVL$ and $(x_\alpha), (y_\alpha)$ and $(z_\alpha)$ be three nets in $X$ such that $x_\alpha\leq y_\alpha\leq z_\alpha$ for all $\alpha$. If $x_\alpha\ptc x$ and $z_\alpha\ptc x$ for $x\in X$ then $y_\alpha\ptc x$.
\end{prop}

\begin{proof}
It follows from Lemma \ref{monoton is convergent} and \cite[Thm.2.28]{AB}.
\end{proof}

\begin{prop}\label{LO are $p$-continuous}
Let $(x_\alpha)_{\alpha \in A}$ and $(y_\beta)_{\beta \in B}$ be two nets in an $LSNVL$ $(X,p,E_\tau)$. If $x_\alpha\ptc x$ and $y_\beta\ptc y$ then $(x_\alpha\vee y_\beta)_{(\alpha,\beta)\in A\times B}\ptc x \vee y$.
\end{prop}

\begin{proof}
Consider the family of Riesz pseudonorms $\{\rho_j\}_{j\in J}$ that generates the topology $\tau$. Since $x_\alpha\ptc x$ and $y_\beta\ptc y$, we have $p(x_\alpha-x)\tc 0$ and $p(y_\beta-y)\tc 0$ in $E$, or $\rho_j\big(p(x_\alpha-x)\big)\to 0$, and also $\rho_j\big(p(y_\beta-y)\big)\to 0$ in $\mathbb{R}$ for all $j\in J$. By using \cite[Thm.1.9(2)]{ABPO}, we have
\begin{eqnarray*}
p(x_\alpha \vee y_\beta - x\vee y)&\leq& p(\lvert x_\alpha \vee y_\beta -x_\alpha \vee y\rvert)+p(\lvert x_\alpha \vee y- x\vee y\rvert)\\&\leq&p(\lvert y_\beta-y\rvert)+p(\lvert x_\alpha-x\rvert).
\end{eqnarray*}
Thus, we have $\rho_j\big(p(x_\alpha \vee y_\beta-x\vee y)\big)\leq\rho_j\big(p( y_\beta-y)\big)+\rho_j\big(p(x_\alpha-x)\big)$ for all $j\in J$. Hence, we get $\rho_j\big(p(x_\alpha \vee y_\beta-x\vee y)\big) \to 0$ in $\mathbb{R}$ for all $j\in J$. Therefore, $(x_\alpha\vee y_\beta)_{(\alpha,\beta)\in A\times B}\ptc x \vee y$ in $X$.
\end{proof}

\begin{defn}
Let $(X,p,E_\tau)$ be an $LSNVL$. A subset $A$ of $X$ is called $p_\tau$-closed set in $X$ if, for any net $(x_\alpha)$ in $A$ that is $p_\tau$-convergent to $x\in X$, these holds $x\in A$.
\end{defn}
It is clear that the positive cone $X_+$ of an $LSNVL$ $X$ is $p_\tau$-closed. Indeed, assume $(x_\alpha)$ is a net in $X_+$ such that it $p_\tau$-converges to $x\in X$. By Proposition \ref{LO are $p$-continuous}, we have $x_\alpha=x_\alpha^+\ptc x^+$, and so we get $x= x^+$. Therefore, $x\in X_+$.
\begin{thm}\label{monoton is order conv}
Any monotone $p_\tau$-convergent net in an $LSNVL$ $(X,p,E_\tau)$ order converges to its $p_\tau$-limits.
\end{thm}

\begin{proof}
It is enough to show that if a net $(x_\alpha)$ is increasing and $p_\tau$-convergent to $x\in X$ then $x_\alpha\uparrow x$. Fix an arbitrary index $\alpha$. Then $x_\beta-x_\alpha\in X_+$ for $\beta\ge\alpha$. Since $X_+$ is $p_\tau$-closed, we have $x_\beta-x_\alpha\xrightarrow[\beta]{p_\tau}x-x_\alpha\in X_+$; see Lemma \ref{basic properties}$(iii)$, and so we get $x\geq x_\alpha$. Since $\alpha$ is arbitrary, $x$ is an upper bound of $x_\alpha$. By $p_\tau$-closeness of $X_+$, if $y\geq x_\alpha$ for all $\alpha$ then $y-x_\alpha\xrightarrow{p} y-x\in X_+$. Therefore, $y\ge x$ and so $x_\alpha \uparrow x$.
\end{proof}

We continue with several basic notions in an $LSNVL$ which are motivated by their analogues vector lattice theory.
\begin{defn}\label{$p$-notions}
Let $X=(X,p,E_\tau)$ be an $LSNVL$. Then 
\begin{enumerate}
\item[(1)] a net $(x_\alpha)_{\alpha \in A}$ in $X$ is said to be {\em $p_\tau$-Cauchy} if the net $(x_\alpha-x_{\alpha'})_{(\alpha,\alpha') \in A\times A}\ptc 0$,
	
\item[(2)] $X$ is called {\em $p_\tau$-complete} if every $p_\tau$-Cauchy net in $X$ is $p_\tau$-convergent,
	
\item[(3)] $X$ is called {\em $op_\tau$-continuous} if $x_\alpha\xrightarrow{o}0$ implies $p(x_\alpha)\xrightarrow{\tau}0$.
\end{enumerate}	
\end{defn}

\begin{rem}\label{ptau closed}
A $p_\tau$-closed sublattice in an $op_\tau$-continuous $LSNVL$ is order closed. Indeed, suppose that $Y$ is $p_\tau$-closed in $X$, $(y_\alpha)$ is a net in $Y$ and $x\in X$ such that $y_\alpha\oc x$. Since $X$ is $op_\tau$-continuous, we have $y_\alpha \ptc x$. Thus, since $Y$ is $p_\tau$-closed, we get $x\in Y$.
\end{rem}

Recall that a locally solid vector lattice $(X,\tau)$ is said to have the {\em Lebesgue property} if $x_\alpha \oc 0$ implies $x_\alpha \tc 0$. It is clear that the $LSNVL$ $(X,\lvert \cdot \rvert,X_\tau)$ with $(X,\tau)$ having the Lebesgue property is $op_\tau$-continuous. 
\begin{prop}\label{$op$-cont-0}
For an $LSNVL$ $(X,p,E_\tau)$ with $(E,\tau)$ having the Lebesgue property,
\begin{enumerate}
\item[(i)] the $p$-convergence implies the $p_\tau$-convergence,
		
\item[(ii)] if, for any net $(x_\alpha)$ in $X$, $x_\alpha\downarrow 0$ implies $p(x_\alpha)\downarrow 0$ then $X$ is $op_\tau$-continuous. 
\end{enumerate}
\end{prop}

\begin{prop}\label{lebesgue implies Cauchy}
For an $op_\tau$-continuous $LSNVL$ $(X,p,E_\tau)$, if $0\leq x_\alpha\uparrow\leq x$ holds in X then $(x_\alpha)$ is a $p_\tau$-Cauchy net in $X$.
\end{prop}

\begin{proof}
Let $0\leq x_\alpha\uparrow\leq x$ in X. Then there exists a net $(y_\beta)$ in X such that $(y_\beta-x_\alpha)_{\alpha,\beta}\downarrow 0$; see \cite[Lem.4.8]{ABPO}. Thus, by the $op_\tau$-continuity, we get $p(y_\beta-x_\alpha)\tc 0$, and so 
$$
p(x_\alpha-x_{\alpha^{'}})_{(\alpha,\alpha^{'})\in (AXA)}\leq p(x_\alpha-y_\beta)+p(y_\beta-x_{\alpha^{'}})\tc0.
$$
Therefore, we get $p(x_\alpha-x_{\alpha^{'}})_{(\alpha,\alpha^{'})\in (AXA)}\tc 0$, so the net $(x_\alpha)$ is a $p_\tau$-Cauchy.
\end{proof}

\begin{cor}\label{op + p implies o}
Let $(X,p,E_\tau)$ be an $op_\tau$-continuous $LSNVL$. If $X$ is $p_\tau$-complete then it is order complete. 
\end{cor}

\begin{proof}
Assume $0\leq x_\alpha\uparrow\leq u$. Then, by Proposition \ref{lebesgue implies Cauchy}, $(x_\alpha)$ is a $p_\tau$-Cauchy net. Since $X$ is $p_\tau$-complete, there is $x\in X$ such that $x_\alpha\ptc x$. It follows from Theorem \ref{monoton is order conv} that $x_\alpha\uparrow x$, and so $X$ is order complete.
\end{proof}


\section{The $up_\tau$-Convergence and $p_\tau$-Unit}

The $up_\tau$-convergence in $LSNVL$s generalizes the $up$-convergence in lattice-normed vector lattices \cite{AEEM}, the $uo$-convergence in vector lattices \cite{GX,GTX}, and the $un$-convergence \cite{DOT}. For a locally solid vector lattice $(X,\tau)$, a net $(x_\alpha)$ in $X$ is called unbounded $\tau$-convergent to $x\in X$ if, for any $u\in X_+$, $\lvert x_\alpha - x\rvert \wedge u \xrightarrow{u\tau}0$. This is written as $x_\alpha \xrightarrow{u\tau}x$ and say $(x_\alpha)$ $u\tau$-converges to $x$. Obviously, $x_\alpha \xrightarrow{\tau}x$ implies $x_\alpha \xrightarrow{u\tau}x$. The converse holds for order bounded nets; see \cite{DEM}.

\begin{defn}\label{$up$-convergence}
Let $(X,p,E_\tau)$ be an $LSNVL$. Then a net $(x_\alpha)$ in $X$ is said to be {\em unbounded $p_\tau$-convergent} to $x$ (or, $(x_\alpha)$ $up_\tau$-converges to $x$, or $x_\alpha\uptc x$) if, for all $u\in X_+$, $	p(\lvert x_\alpha-x\rvert \wedge u)\xrightarrow{\tau}0$.
\end{defn}
By the inequality $\lvert x_\alpha-x\rvert \wedge u \leq \lvert x_\alpha-x\rvert$ for all $u\in X_+$ and for all $\alpha$, it can be seen that the $p_\tau$-convergence implies the $up_\tau$-convergence. The following result is an $up_\tau$-version of \cite[Cor.3.6]{GTX}.
\begin{prop}
A disjoint sequence $(x_n)$ in a sequentially $op_\tau$-continuous $LSNVL$ $(X,p,E_\tau)$ is sequentially $up_\tau$-convergent to zero. 
\end{prop}

One can define a $up_\tau$-closed subset as following: let $(X,p,E_\tau)$ be an $LSNVL$ and $Y$ be a sublattice of $X$. Then $Y$ is called {\em $up_\tau$-closed} in $X$ if, for any net $(y_\alpha)$ in $Y$ that is $up_\tau$-convergent to $x\in X$, we have $x\in Y$. So, it is clear that every band is $up_\tau$-closed, and that an $up_\tau$-closed sublattice in an $op_\tau$-continuous $LSNVL$ is $uo$-closed. Similar to Theorem \ref{monoton is order conv}, the next result can be observed.
\begin{prop}\label{up monotone order conv}
Any monotone and $up_\tau$-convergence net in an $LSNVL$ $(X,p,E_\tau)$ is order convergent to its $up_\tau$-limit.
\end{prop}

It is known that the $p_\tau$-convergence implies the $up_\tau$-convergence, but for converse, we generalize a $p_\tau$-version of \cite[Lem.1.2(ii)]{KMT} in the following theorem.
\begin{thm}\label{up implies p conv.}
Assume that $(x_\alpha)$ is a monotone net in an $LSNVL$ $(X,p,E_\tau)$ such that $x_\alpha\uptc x$ in $X$. Then $x_\alpha \ptc x$.
\end{thm}

\begin{proof}
We may assume that, without loss of generality, $x_\alpha$ is increasing and $0\le x_\alpha$ for all $\alpha$. From Proposition \ref{up monotone order conv}, it follows that 
$0\le x_\alpha\uparrow x$ for some $x\in X$ since $x_\alpha$ is $up_\tau$-convergent. So, $0\leq x-x_\alpha\leq x$ for all $\alpha$. For each $u\in X_+$, we have $p\big((x-x_\alpha)\wedge u\big)\tc 0$. In particular, for $u=x$, $(x-x_\alpha)\wedge x=x-x_\alpha$, and so we obtain that $p(x-x_\alpha)=p\big((x_\alpha-x)\wedge x\big)\tc 0$. Therefore, $x_\alpha \ptc x$.
\end{proof}	

\begin{defn}
Let $(X,p,E_\tau)$ be an $LSNVL$. A subset $Y$ of $X$ is called $p_\tau$-bounded if $p(Y)$ is $\tau$-bounded in $E$.
\end{defn}

It is clear that if a net $(x_\alpha)_{\alpha \in A}$ in an $LSNVL$ $(X,p,E_\tau)$ is $p_\tau$-convergent then it is $p_\tau$-bounded. Also, using \cite[Thm.2.19(i)]{AB}, it can be seen that a $p$-bounded net is $p_\tau$-bounded.

\begin{prop}
Let $(X,p,E_\tau)$ be an $LSNVL$ with $(E,\tau)$ having an order bounded $\tau$-neighborhood of zero. Thus, if a net $(x_\alpha)$ is $p_\tau$-bounded in $X$ then it is $p$-bounded.
\end{prop}

\begin{proof}
Since the net $(x_\alpha)$ is a $p_\tau$-bounded, $p(x_\alpha)$ is a $\tau$-bounded net in $E$. By \cite[Thm.2.2]{L}, $p(x_\alpha)$ is also order bounded in $E$. Therefore, $x_\alpha$ is $p$-bounded in $X$.
\end{proof} 

We now turn our attention to the $p_\tau$-unit, which was introduced for the order convergence as the $p$-unit in \cite{AEEM}. That notion was motivated by the notion of a weak order unit in a vector lattice $X=(X,\lvert\cdot\rvert,X)$ and by the notion of a quasi-interior point in a normed lattice $X=(X,\lVert\cdot\rVert,{\mathbb R})$.
\begin{defn}
Let $(X,p,E_\tau)$ be an $LSNVL$. A vector $e\in X$ is called a {\em $p_\tau$-unit} if, for any $x\in X_+$, $p(x-ne\wedge x)\xrightarrow{\tau}0$.
\end{defn}

Let $e$ be a $p_\tau$-unit in an $LSNVL$ $(X,p,E_\tau)$, where  $X\ne\{0\}$. So, it holds that $e>0$. Indeed, let $e\ne 0$. Suppose $e^->0$. Then, for $x:=e^-$, we obtain
$$
p(x-x\wedge ne)=p(e^--(e^-\wedge n(e^+-e^-)))= p(e^--(e^-\wedge n(-e^-)))=(n+1)p(e^-).
$$
As $n\to\infty$, by using \cite[Thm.2.28]{AB}, we have $(n+1)p(e^-)\not\xrightarrow{\tau}0$. This is impossible, because $e$ is a $p_\tau$-unit. Therefore, $e^-=0$ and so $e>0$. In the following result, we give a generalization of \cite[Lem.2.11]{DOT} and \cite[Cor.3.5]{GTX}. 
\begin{thm}\label{$up$-conv by $p$-unit}
Let $(X,p_\tau,E)$ be an $LSNVL$ and $e\in X$ be a $p_\tau$-unit. Then, $ x_{\alpha} \stackrel{up_\tau}\to 0 $ iff $p(|x_\alpha | \wedge e)\stackrel{\tau}\to 0$.
\end{thm}

\begin{proof}
The forward implication is immediate. For the reverse implication, let $u\in X_+$ be arbitrary. Note that
$$
\left| x_\alpha\right|\wedge u\leq \left| x_\alpha\right|\wedge(u-u\wedge ne)+\left| x_\alpha\right|\wedge(u\wedge ne)\leq (u-u\wedge ne)+n(\left| x_\alpha\right|\wedge e).
$$
Hence, we get 
$$
p(\left| x_\alpha\right|\wedge u)\leq p(u-u\wedge ne)+np(\left| x_\alpha\right|\wedge e)
$$
for all $\alpha$ and all $n\in \mathbb{N}$. Consider the family of Riesz pseudonorms $\{\rho_j\}_{j\in J}$ which generates the topology $\tau$. So, we have 
$$
\rho_j\big(p(\left| x_\alpha\right|\wedge u)\big)\leq \rho_j\big(p(u-u\wedge ne)\big)+n\rho_j\big(p(\left| x_\alpha\right|\wedge e)\big)
$$
for all $j\in J$. Fix $\varepsilon>0$. Since $e$ is $p_\tau$-unit, it follows from $p(u-u\wedge ne)\tc 0$ that there exists $n_0\in \mathbb{N}$ such that $\rho_j\big(p(u-u\wedge ne)\big)<\varepsilon$ for all $n>n_0$. Furthermore, it follows from $p(|x|\wedge e)\tc 0 $ that there is $\alpha_0$ such that $\rho_j\big(p(\left| x_\alpha\right|\wedge e)\big)<\frac{\varepsilon}{n}$ whenever $\alpha>\alpha_0$. It follows that
$$
\rho_j\big(p(\left| x_\alpha\right|\wedge u)\big)<\varepsilon+n\frac{\varepsilon}{n}=2\varepsilon.
$$
Therefore, $\rho_j\big(p(\left| x_\alpha\right|\wedge u)\big)\to 0$, or $p(\left| x_\alpha\right|\wedge u)\tc 0$.
\end{proof}

\begin{rem}\label{properties of $p$-units}
Let $(X,p,E_\tau)$ be an $LSNVL$.
\begin{enumerate}
\item[(1)] If $e\in X$ is a strong unit, then $e$ is a $p_\tau$-unit.
		
\item[(2)] If $e\in X$ is a $p_\tau$-unit, then $e$ is a weak unit. 
		
\item[(3)] Let $e\in X$ be a $p_\tau$-unit. For any $0<\alpha\in\mathbb{R}_+$ and $z\in X_+$, $\alpha e$ and $e+z$ are both $p_\tau$-units.
		
\item[(4)] If $X$ is $op_\tau$-continuous, then clearly every weak unit of $X$ is a $p_\tau$-unit.
		
\item[(5)] In an $LSNVL$ $(E,\lvert\cdot\rvert,E_\tau)$ with $(E,\tau)$ having the Lebesgue property, the lattice norm $p(x)=|x|$ is always $op_\tau$-continuous. Therefore, the notions of $p_\tau$-unit and of weak unit coincide in $E$.
		
\item[(6)] Let $(X,\lVert\cdot\rVert)$ be a normed lattice. Then, in $LSNVL$ $(X,\lVert\cdot\rVert,\mathbb{R}_{\lvert\cdot\rvert})$, $e\in X$ is a $p_\tau$-unit iff $e$ is a quasi-interior point of $X$.
\end{enumerate}
\end{rem}

The following theorem is an analogue of \cite[Prop.8]{AEEM} and it has the same technique as in the proof of \cite[Lem.4.15]{AA}.
\begin{thm}
Let $(X,p,E_\tau)$ be an $LSNVL$, $e\in X_+$ and $I_e$ be the order ideal generated by $e$ in $X$. If $I_e$ is $p$-dense in $X$, then $e$ is a $p_\tau$-unit.
\end{thm}

\begin{proof}
Let $x\in X_+$, then there is $y\in I_e$ such that, for any $0\neq u\in p(X)$, $p(x-y)\leq u$. From the Birkhoff's inequality $|y^+\wedge x-x|\leq |y^+-x|=|y^+-x^+|\leq |y-x|$, by replacing $y$ by $y^+\wedge x$, we see that there exists $y\in I_e$ such that $0\leq y\leq x$ and $p(x-y)\leq u$. Thus, for any $k\in\mathbb{N}$, there is $y_k\in I_e$ such that $0\leq y_k\leq x$ and $p(x-y_k) \leq \frac{1}{k}u$. Since $y_k\in I_e$, then there exists $m=m(k)\in\mathbb{N}$ such that $0\leq y_k\leq me$, and so $0\leq y_k\leq me\wedge x$. For any $n\geq m$, we have $x-x\wedge ne\leq x-x\wedge me\leq x-y_k$, and so $p(x-x\wedge ne)\leq p(x-y_k)\leq\frac{1}{k}u$. Hence $p(x-x\wedge ne)\xrightarrow{\tau}0$, and so $e$ is a $p_\tau$-unit.
\end{proof}


\section{Convergences in Sublattices}
The $up_\tau$-convergence passes obviously to sublattices of vector lattices as it was remarked in \cite[p.3]{DOT}. In opposite to the $uo$-convergence \cite[Thm.3.2]{GTX}, the $un$-convergence does not pass even from regular sublattice. Let $Y$ be a sublattice of an $LSNVL$ $X=(X,p,E_\tau)$  and $(y_\alpha)_{\alpha\in A}$ be a net in $Y$. Then we can define the $up_\tau$-convergence in $Y$ as following: $y_\alpha \uptc 0$ iff $p(\lvert y_\alpha\rvert\wedge u)\tc 0$ for all $u\in Y_+$. It is clear that if a net in a sublattice $Y$ of an $LSNVL$ $(X,p,E_\tau)$ $up_\tau$-converges to zero in $X$ then it also $up_\tau$-converges to zero in $Y$. For the converse, we give the following theorem.

\begin{thm}\label{$up$-regular}
Let $Y$ be a sublattice of an $LSNVL$ $(X,p,E_\tau)$ and $(y_\alpha)$ be a net in $Y$ such that $y_\alpha\uptc 0$ in $Y$. Thus, $y_\alpha\uptc 0$ in $X$ for each of the following cases:
	
$(i)$\ $Y$ is majorizing in $X$;
	
$(ii)$\ $Y$ is p-dense in $X$;
	
$(iii)$\ $Y$ is a projection band in $X$.
\end{thm}

\begin{proof}
Assume $(y_\alpha)\subseteq Y$ is a net such that $y_\alpha\uptc 0$ in $Y$. Take any vector $0\neq x\in X_+$.
	
$(i)$ Since  $Y$ is majorizing in $X$, there exists $y\in Y$ such that $x\leq y$. It follows from $0\leq \lvert y_\alpha\rvert\wedge x\leq \lvert y_\alpha\rvert\wedge y \ptc 0$, we have $y_\alpha\uptc 0$ in $X$.
	
$(ii)$ Choose an arbitrary $0\ne u\in p(X)$. Then there exists $y\in Y$ such that $p(x-y)\le u$. Since $\lvert y_\alpha\rvert\wedge x\le \lvert y_\alpha\rvert\wedge \lvert x-y\rvert+\lvert y_\alpha\rvert\wedge \rvert y\rvert$, we have
\begin{eqnarray*}
p(\lvert y_\alpha\rvert\wedge x)&\leq& p(\lvert y_\alpha\rvert\wedge \lvert x-y\rvert)+p(\lvert y_\alpha\rvert\wedge \lvert y\rvert) \\ &\leq& p(\lvert x-y\rvert)+p(\lvert y_\alpha\rvert\wedge \lvert y\rvert) \\ &\leq& u+p(\lvert y_\alpha\rvert\wedge \lvert y\rvert).
\end{eqnarray*}
Also, since $0\ne u\in p(X)$ is arbitrary and $p(\lvert y_\alpha\rvert\wedge \lvert y\rvert)\tc 0$, we get $y_\alpha\uptc 0$ in $X$.
	
$(iii)$ Suppose that $Y$ is a projection band in $X$. Thus, $X=Y\oplus Y^{\bot}$. Hence $x=x_1+x_2$ with $x_1\in Y$ and $x_2\in Y^{\bot}$. Since $y_\alpha\wedge x_2=0$, we have 
$$
p(y_\alpha\wedge x)=p(y_\alpha\wedge(x_1+x_2))\leq p(y_\alpha\wedge x_1)+p(y_\alpha\wedge x_2)=p(y_\alpha\wedge x_1) \tc 0.
$$
So, we get $y_\alpha\uptc 0$ in $X$.
\end{proof}
Recall that every Archimedean vector lattice $X$ is majorizing in its order completion $\tilde{X}$; see \cite[p.101]{ABPO}. Thus, the following result arises.
\begin{cor}\label{order completion}
Let $(X,p,E_\tau)$ be an $LSNVL$ and $(x_\alpha)$ be a net in $X$ such that $x_\alpha\uptc 0$ in $X$. Then $x_\alpha\uptc 0$ in the $(\tilde{X},p,E_\tau)$. 
\end{cor}

We continue with a version of \cite[Prop.3.15]{GTX}.
\begin{thm}\label{up closed and ptau closed}
Let $(X,p,E_\tau)$ be an $op_\tau$-continuous $LSNVL$ and $Y$ be a sublattice of $X$. Then $Y$ is $up_\tau$-closed in $X$ iff it is $p_\tau$-closed in $X$.
\end{thm}

\begin{proof}
The forward implication is obvious. For the converse, suppose that $Y$ is $p_\tau$-closed in $X$. Let $(y_\alpha)$ be a net in $Y$ and $x\in X$ such that $y_\alpha \uptc x$ in $X$. Then,  without loss of generality, we assume $y_\alpha$ in $Y_+$ for each $\alpha$ and $x\in X_+$. Note that, for every $z\in X_+$, $\lvert y_\alpha\wedge z-x\wedge z\rvert\leq \lvert y_\alpha-x\rvert\wedge z$ \ (cf. the inequality (1) in the proof of \cite[Prop.3.15]{GTX}). Hence, $p(y_\alpha\wedge z-x\wedge z)\le p(\lvert y_\alpha-x\rvert\wedge z)\tc 0$. Thus, we get $y_\alpha\wedge z\ptc x\wedge z$ for every $z\in X_+$. In particular, $y_\alpha\wedge y\ptc x\wedge y$ for any $y\in Y_+$. Since $Y$ is $p_\tau$-closed, $x\wedge y\in Y$ for any $y\in Y_+$. For any $0\le w\in Y^\perp$ and for any $\alpha$, we have $y_\alpha\wedge w=0$. Then
$$
\lvert x\wedge w\rvert=\lvert y_\alpha\wedge w-x\wedge w\rvert\le \lvert y_\alpha-x\rvert\wedge w\ptc 0.
$$
So, $x\wedge w=0$, and hence $x\in Y^{\perp\perp}$. Since $Y^{\perp\perp}$ is the band generated by $Y$ in $X$, there is a net $(z_\beta)$ in the ideal $I_Y$ generated by $Y$ such that $0\le z_\beta\uparrow x$ in $X$. Choose an element $t_\beta\in Y$ with $z_\beta\le t_\beta$, for each $\beta$.
Then $x\geq t_\beta\wedge x\ge z_\beta\wedge x=z_\beta\uparrow x$ in $X$, and so $t_\beta\wedge x\oc x$ in $X$. Since $X$ is $op_\tau$-continuous, $t_\beta\wedge x\ptc x$. So, as $t_\beta\wedge x\in Y$ and $Y$ is $p_\tau$-closed, we get $x\in Y$.
\end{proof}

\begin{rem}\label{monotonosity of pousedunorm}
If $\rho$ is a Riesz pseudonorm on a vector lattice $X$ and $x\in X$ then $\rho(\frac{1}{n}x)\leq \frac{1}{n}\rho(x) $ for all $n\in\mathbb{N}$; see \cite[p.2]{DEM}. 
\end{rem}

Recall that a subset $A$ of an LNVL $(X,p,E)$ is called {\em $p$-almost order bounded} if, for any $w\in E_+$, there is $x_w\in X_+$ such that $p((\lvert x\rvert-x_w)^+)=p(\lvert x\rvert-x_w\wedge \lvert x\rvert)\leq w$ for any $x\in A$. In the following theorem, we prove that the $up_\tau$-convergence implies the $p_\tau$-convergence, which is a $p$-version of \cite[Lem.2.9]{DOT}, and it is also similar to \cite[Prop.3.7]{GX} and \cite[Prop.9]{AEEM}.
\begin{thm}\label{$p$-almost order boundedness and $up$-convergence}
Let $(X,p,E_\tau)$ be an LSLNVL. If $(x_\alpha)$ is $p$-almost order bounded and $x_\alpha\uptc x$ then $x_\alpha\ptc x$.
\end{thm}

\begin{proof}
Suppose that $(x_\alpha)$ is $p$-almost order bounded. Then the net $(\lvert x_\alpha-x\rvert)_\alpha$ is also $p$-almost order bounded. So, for arbitrary $w\in E_+$ and $n\in \mathbb{N}$, there exists $x_w\in X_+$ with 
$$ 
p(\lvert x_\alpha-x\rvert-\lvert x_\alpha-x\rvert\wedge x_w)=p\big((\lvert x_\alpha-x\rvert-x_w)^+\big)\leq \frac{1}{n}w.
$$
Since $x_\alpha \uptc x$, we have $ p(\lvert x_\alpha-x\rvert\wedge x_w)\tc 0$ in $E$. Thus, for the family of Riesz pseudonorms $\{\rho_j\}_{j\in J}$ that generates the topology $\tau$, we have $\rho_j\big(p(\lvert x_\alpha -x\rvert\wedge x_w)\to 0$ for all $j\in J$. Moreover, for any $\alpha$,
\begin{eqnarray*}
p(x_\alpha -x)=p(\lvert x_\alpha -x\rvert)&\leq& p(\lvert x_\alpha -x\rvert-\lvert x_\alpha -x\rvert \wedge x_w)+p(\lvert x_\alpha -x\rvert\wedge x_w)\\ &\leq& \frac{1}{n}w + p(\lvert x_\alpha -x\rvert\wedge x_w). 
\end{eqnarray*}
Hence, by Remark  \ref{monotonosity of pousedunorm}, we get
\begin{eqnarray*}
\rho_j\big(p(x_\alpha -x)\big)&\leq& \rho_j\big( \frac{1}{n}w+p(\lvert x_\alpha -x\rvert\wedge x_w)\big) \ \ \ (\forall j)\\ &\leq& \rho_j(\frac{1}{n}w)+ \rho_j\big(p(\lvert x_\alpha -x\rvert\wedge x_w)\big)\to  \rho_j(\frac{1}{n}w)\leq \frac{1}{n} \rho_j(w)\to 0.
\end{eqnarray*}

Therefore, $\rho_j\big(p(x_\alpha -x)\big)\to 0$, and so $x_\alpha\ptc x$.
\end{proof}

\section{The Mixed-Normed Spaces}
Let $(X,p,E)$ be an $LNS$ and $(E,\lVert\cdot\rVert_E)$ be a normed lattice. The \textit{mixed norm} on $X$ is defined by $p\text{-}\lVert x\rVert_E=\lVert p(x)\rVert_E$ for all $x\in X$. In this case the normed space $(X,p\text{-}\lVert\cdot\rVert)$ is called a \textit{mixed-normed space}; see, for example \cite[7.1.1, p.292]{K}. Consider an LNS $(X,p,E_\tau)$ with $\tau$ being a $un$-topology and $(E,\lVert\cdot\rVert_E)$ being Banach lattice. If $E$ has a strong unit then $x_\alpha\xrightarrow{p\text{-}\lVert \cdot\rVert_E} 0$ iff $x_\alpha\ptc 0$ for any net $(x_\alpha)$ in $X$; see \cite[Thm.2.3]{KMT}. We give results about $uo$- and $un$- topology in the following remark.

\begin{rem}
Let $(X,p,E_\tau)$ be an $op_\tau$-continuous $LSNVL$ with $(E,\lVert\cdot\rVert)$ being normed lattice.
\begin{enumerate}
\item[(i)] If $\tau$ is a $uo$-topology and $(E,\lVert\cdot\rVert)$ is an order continuous normed lattice then $(X,p\text{-}\lVert\cdot\rVert)$ is an order continuous normed lattice. Indeed, suppose $x_\alpha\downarrow 0$ in $X$. Since $(X,p,E_\tau)$ is $op_\tau$-continuous, $p(x_\alpha)\tc 0$ in $E$. Also, since $p(x_\alpha)\downarrow$ and $p(x_\alpha)\uoc 0$, by \cite[Lem.1.2(1)]{KMT}, we have $p(x_\alpha)\downarrow 0$. So, by using the order continuity of norm, we have $\lVert p(x_\alpha)\rVert \downarrow 0$ or $p\text{-}\lVert x_\alpha\rVert\downarrow 0$.
		
\item[(ii)] If $\tau$ is a $un$-topology then $(X,p\text{-}\lVert\cdot\rVert)$ is an order continuous normed lattice. Indeed, assume $x_\alpha\downarrow 0$ in $X$. Hence, $p(x_\alpha)\tc 0$ in $E$ because $(X,p,E_\tau)$ is $op_\tau$-continuous. Since $p(x_\alpha)\unc 0$ and $p(x_\alpha)\downarrow$, applying \cite[Lem.1.2(2)]{KMT}, we get $\lVert p(x_\alpha)\rVert \downarrow 0$ or $p\text{-}\lVert x_\alpha\rVert\downarrow 0$.
\end{enumerate}
\end{rem}

Recall that a Banach lattice is called $un$-complete if every $un$-Cauchy net is un-convergent, \cite{KMT}. Similarly, one can define $up_\tau$-complete. 
\begin{thm}
Let $(X,p,E_\tau)$ be an $LSNVL$ with $\tau$ being the norm-topology on a normed lattice $(E,\left\|\cdot\right\|)$. If $(X,p\text{-}\left\|\cdot\right\|)$ is a $un$-complete Banach lattice, then $X$ is $up_\tau$-complete. 
\end{thm}

\begin{proof}
Let $(x_\alpha)$ be a $up_\tau$-Cauchy net in $X$. So, for every $u\in X_+$, $p(|x_\alpha-x_\beta|\wedge u)\xrightarrow{\tau}0$, or for every $u\in X_+$, $\left\|p(|x_\alpha-x_\beta|\wedge u)\right\|\to 0$ or for every $u\in X_+$, $p\text{-}\left\| |x_\alpha-x_\beta|\wedge u\right\|\to 0$, i.e. $x_\alpha$ is $un$-Cauchy in $(X, p\text{-}\left\|\cdot\right\|)$. Since $(X,p\text{-}\left\|\cdot\right\|)$ is $un$-complete, then there exists $x\in X$ such that $x_\alpha\xrightarrow{un}x$ in $(X,p\text{-}\left\|\cdot\right\|)$. That is, for every $u\in X_+$, $\left\|p(|x_\alpha-x|\wedge u)\right\|\to 0$, i.e. $x_\alpha\uptc x$.
\end{proof} 

Let $X=(X,p,E_\tau)$ be an $LSNVL$. Then $X$ is called a {\em $p_\tau$-KB-space} if every $p_\tau$-bounded increasing net in $X_+$ is $p_\tau$-convergent.
\begin{thm}\label{pKB-pKB}
Let $(X,p,E_\mu)$ and $(E,m,F_\tau)$ be two $p_\tau$-$KB$-spaces. Then the $LSNVL$ $(X,m\circ p,F_\tau)$ is also a $p_\tau$-KB-space.
\end{thm}

\begin{proof}
Let $0\leq x_\alpha\uparrow$, and $m(p[x_\alpha])$ is $\tau$-bounded or $p(x_\alpha)$ is $p_\tau$-bounded in $F$. Hence, since $0\leq p(x_\alpha)\uparrow<\infty$ and $(E,m,F_\mu)$ is a $p_\tau$-KB-space, there exists $y\in E$ such that $m(p(x_\alpha)-y))\tc 0$ in $F$. Thus, $p(x_\alpha)\ptc y$ in $E$, and so the net $(x_\alpha)$ is increasing and $p_\tau$-bounded. Since $X$ is $p_\tau$-$KB$-space, then there exists $x\in X$ such that $p(x_\alpha-x)\to 0$. It can be clearly seen that $(X,m\circ p,F)$ is $op_\tau$-continuous, and so $m(p(x_\alpha-x))\xrightarrow{\tau}0$ i.e. $(m\circ p)[x_\alpha-x]\xrightarrow{\tau}0$. Thus $(X,m\circ p,F)$ is a $p_\tau$-KB-space.
\end{proof}

\begin{cor}\label{pKB-KB}
Let $(X,p,E_\tau)$ be a $p_\tau$-$KB$-space and $(E,\lVert\cdot\rVert)$ be a KB-space. Then $(X,p-\lVert\cdot\rVert)$ is a KB-space.
\end{cor}

\end{document}